\documentclass[12pt]{article}
\usepackage{amssymb,bm}
\usepackage{amsmath}
\usepackage{amsthm}
\usepackage{graphicx}
\usepackage[active]{srcltx} 
\usepackage{hyperref} 
\hypersetup{pdfborder=0 0 0}
\newtheorem{theorem}{{\sc Theorem}}[section]

\newtheorem{lemma}[theorem]{{\sc Lemma}}

\newtheorem{remark}[theorem]{Remark}

\newcommand{\bb}[1]{\mathbb{ #1}}

\newcommand{\Skew}{\mathrm{Skew}}

\newcommand{\hf}{\displaystyle\frac{1}{2}}
\newcommand{\nth}[1]{\displaystyle\frac{1}{#1}}

\newcommand{\grad}{\mathrm{grad}\,}
\newcommand{\dif}[2]{\displaystyle\frac{\partial #1}{\partial #2}}
\newcommand{\Grad}{\nabla}

\newcommand{\Md}{\partial}

\def\XXint#1#2#3{{\setbox0=\hbox{$#1{#2#3}{\int}$ }
\vcenter{\hbox{$#2#3$ }}\kern-.6\wd0}}

\newcommand{\mat}[4]{\left[\begin{array}{cc}
\displaystyle{#1}&\displaystyle{#2}\\[1ex]
\displaystyle{#3}&\displaystyle{#4}\end{array}\right]}

\newcommand{\bc}{boundary condition}
\newcommand{\bvp}{boundary value problem}

\newcommand{\IFF}{if and only if }

\newcommand{\Ga}{\alpha}
\newcommand{\Gb}{\beta}

\newcommand{\Ge}{\epsilon}

\newcommand{\Gg}{\gamma}

\newcommand{\Gk}{\kappa}

\newcommand{\Gl}{\lambda}

\newcommand{\Gth}{\theta}

\newcommand{\GO}{\Omega}

\bmdefine\BGa{\alpha}
\bmdefine\BGb{\beta}
\bmdefine\BGd{\delta}
\bmdefine\BGe{\epsilon}
\bmdefine\BGve{\varepsilon}
\bmdefine\BGf{\phi}
\bmdefine\BGvf{\varphi}
\bmdefine\BGg{\gamma}
\bmdefine\BGc{\chi}
\bmdefine\BGi{\iota}
\bmdefine\BGk{\kappa}
\bmdefine\BGl{\lambda}
\bmdefine\BGn{\eta}
\bmdefine\BGm{\mu}
\bmdefine\BGv{\nu}
\bmdefine\BGp{\pi}
\bmdefine\BGth{\theta}
\bmdefine\BGvth{\vartheta}
\bmdefine\BGr{\rho}
\bmdefine\BGvr{\varrho}
\bmdefine\BGs{\sigma}
\bmdefine\BGvs{\varsigma}
\bmdefine\BGt{\tau}
\bmdefine\BGj{\tau}
\bmdefine\BGu{\upsilon}
\bmdefine\BGo{\omega}
\bmdefine\BGx{\xi}
\bmdefine\BGy{\psi}
\bmdefine\BGz{\zeta}
\bmdefine\BGD{\Delta}
\bmdefine\BGF{\Phi}
\bmdefine\BGG{\Gamma}
\bmdefine\BGL{\Lambda}
\bmdefine\BGP{\Pi}
\bmdefine\BGT{\Theta}
\bmdefine\BGS{\Sigma}
\bmdefine\BGU{\Upsilon}
\bmdefine\BGO{\Omega}
\bmdefine\BGX{\Xi}
\bmdefine\BGY{\Psi}

\newcommand{\CC}{{\mathcal C}}

\newcommand{\CP}{{\mathcal P}}

\bmdefine\BCA{{\mathcal A}}
\bmdefine\BCB{{\mathcal B}}
\bmdefine\BCC{{\mathcal C}}
\bmdefine\BCD{{\mathcal D}}
\bmdefine\BCE{{\mathcal E}}
\bmdefine\BCF{{\mathcal F}}
\bmdefine\BCG{{\mathcal G}}
\bmdefine\BCH{{\mathcal H}}
\bmdefine\BCI{{\mathcal I}}
\bmdefine\BCJ{{\mathcal J}}
\bmdefine\BCK{{\mathcal K}}
\bmdefine\BCL{{\mathcal L}}
\bmdefine\BCM{{\mathcal M}}
\bmdefine\BCN{{\mathcal N}}
\bmdefine\BCO{{\mathcal O}}
\bmdefine\BCP{{\mathcal P}}
\bmdefine\BCQ{{\mathcal Q}}
\bmdefine\BCR{{\mathcal R}}
\bmdefine\BCS{{\mathcal S}}
\bmdefine\BCT{{\mathcal T}}
\bmdefine\BCU{{\mathcal U}}
\bmdefine\BCV{{\mathcal V}}
\bmdefine\BCW{{\mathcal W}}
\bmdefine\BCX{{\mathcal X}}
\bmdefine\BCY{{\mathcal Y}}
\bmdefine\BCZ{{\mathcal Z}}

\bmdefine\Bzr{ 0}
\bmdefine\Ba{ a}
\bmdefine\Bb{ b}
\bmdefine\Bc{ c}
\bmdefine\Bd{ d}
\bmdefine\Be{ e}
\bmdefine\Bf{ f}
\bmdefine\Bg{ g}
\bmdefine\Bh{ h}
\bmdefine\Bi{ i}
\bmdefine\Bj{ j}
\bmdefine\Bk{ k}
\bmdefine\Bl{ l}
\bmdefine\Bm{ m}
\bmdefine\Bn{ n}
\bmdefine\Bo{ o}
\bmdefine\Bp{ p}
\bmdefine\Bq{ q}
\bmdefine\Br{ r}
\bmdefine\Bs{ s}
\bmdefine\Bt{ t}
\bmdefine\Bu{ u}
\bmdefine\Bv{ v}
\bmdefine\Bw{ w}
\bmdefine\Bx{ x}
\bmdefine\By{ y}
\bmdefine\Bz{ z}
\bmdefine\BA{ A}
\bmdefine\BB{ B}
\bmdefine\BC{ C}
\bmdefine\BD{ D}
\bmdefine\BE{ E}
\bmdefine\BF{ F}
\bmdefine\BG{ G}
\bmdefine\BH{ H}
\bmdefine\BI{ I}
\bmdefine\BJ{ J}
\bmdefine\BK{ K}
\bmdefine\BL{ L}
\bmdefine\BM{ M}
\bmdefine\BN{ N}
\bmdefine\BO{ O}
\bmdefine\BP{ P}
\bmdefine\BQ{ Q}
\bmdefine\BR{ R}
\bmdefine\BS{ S}
\bmdefine\BT{ T}
\bmdefine\BU{ U}
\bmdefine\BV{ V}
\bmdefine\BW{ W}
\bmdefine\BX{ X}
\bmdefine\BY{ Y}
\bmdefine\BZ{ Z}

\begin{document}
\title{Korn inequalities for shells with zero Gaussian curvature}
\author{Yury Grabovsky\footnote{Temple University, Philadelphia, PA, yury@temple.edu}
\and Davit Harutyunyan\footnote{University of Utah, Salt Lake City, UT, davith@math.utah.edu}}
\maketitle

\begin{abstract}
  We consider shells with zero Gaussian curvature, namely shells with one principal curvature
  zero and the other one having a constant sign. Our particular interests are shells that are
  diffeomorphic to a circular cylindrical shell with zero
  principal longitudinal curvature and positive circumferential curvature,
  including, for example, cylindrical and conical shells with arbitrary convex
  cross sections. We prove that the best constant in the first Korn inequality
  scales like thickness to the power 3/2 for a wide range of boundary
  conditions at the thin edges of the shell. Our methodology is to prove, for
  each of the three mutually orthogonal two-dimensional cross-sections of the
  shell, a ``first-and-a-half Korn inequality''|a hybrid between the classical
  first and second Korn inequalities. These three two-dimensional inequalities
  assemble into a three-dimensional one, which, in turn, implies the
  asymptotically sharp first Korn inequality for the shell. This work is a
  part of mathematically rigorous analysis of extreme sensitivity of the
  buckling load of axially compressed cylindrical shells to shape imperfections.
\end{abstract}

\section{Introduction}
\label{sec:1}
Classical first and second Korn inequalities have been known to play a central
role in the theory of linear elasticity and recently they have found very
important applications in the problems of buckling of slender structures
\cite{grtr07,grha15,grha16}. Let us recall the classical first and
second Kohn inequalities, that actually date back to 1908,
\cite{korn08,korn09}. To that end we denote
\[
\mathfrak{euc}(n)=\{\Bu:\mathbb R^n\to\mathbb R^n:\Bu(\Bx)=\BA\Bx+\Bb,\
\BA\in\Skew(\bb{R}^{n}),\ \Bb\in\bb{R}^{n}\}
\]
be the set of all infinitesimal motions, i.e., a Lie algebra of the group
of all Euclidean transformations (rigid body motions).  Let $\GO$ be an open
connected subset of $\bb{R}^{n}$ and $\Bu\in W^{1,2}(\GO;\bb{R}^{n})$. We
denote\footnote{We reserve more streamlined notations $\Grad\Bu$ and $e(\Bu)$
  for ``simplified'' gradient and symmetrized gradient, respectively, that
  will be our main characters in the technical part of the paper.} by
$\grad\Bu$ and $(\grad\Bu)_{\rm sym}$ the gradient and the symmetric part of
the gradient, respectively, of a vector field $\Bu$. It is well-known that $(\grad\Bu)_{\rm sym}=0$ in
$\GO$ (in the sense of distributions) \IFF $\Bu\in\mathfrak{euc}(n)$. This is
an immediate consequence of a simple observation (also very well-known) that
all partial derivatives of the gradient $\BG=\grad\Bu$ can be expressed as a
linear combination of partial derivatives of the symmetric part of the
gradient $\BE=(\grad\Bu)_{\rm sym}$:
\[
\dif{G_{jk}}{x^{i}}=\dif{E_{jk}}{x^{i}}+\dif{E_{ij}}{x^{k}}-\dif{E_{ik}}{x^{j}}.
\]

The classical first Korn inequality (e.g., as
stated in \cite{oshyo92}) quantifies this result by describing how large
$(\grad\Bu)_{\rm sym}$ must be if $\Bu$ lies in a closed subspace $V\subset
W^{1,2}(\GO;\bb{R}^{n})$ that has trivial intersection with
$\mathfrak{euc}(n)$. If $\GO$ is a Lipschitz domain, then there exists a
constant $C(\GO,V)$, such that for every $\Bu\in V$
\begin{equation}
\label{1.2}
\|(\grad\Bu)_{\rm sym}\|^2\ge C(\GO,V)\|\grad\Bu\|^2,
\end{equation}
where $\|\cdot\|$ is the $L^{2}$-norm.
The Korn constant $C(\Omega,V)$ measures the distance between the subspace $V$ and
$\mathfrak{euc}(n)$.
The classical second Korn inequality asserts that the standard $W^{1,2}$ norm
topology can be equivalently defined by replacing $\grad\Bu$ with $(\grad\Bu)_{\rm sym}$:
\begin{equation}
\label{1.1}
\|\grad\Bu\|^2\le C(\GO)(\|(\grad\Bu)_{\rm sym}\|^2+\|\Bu\|^2),\quad\Bu\in W^{1,2}(\GO;\bb{R}^{n}).
\end{equation}
Originally, Korn inequalities were used to prove existence, uniqueness and
well-posedness of \bvp s of linear elasticity (see e.g.,
\cite{love27,ciar00}). Nowadays, often, as in our particular case, it is the
best Korn constant $C(\GO,V)$ in the first Korn inequality that is of central
importance (e.g., \cite{cot89,naza04,naza08,pato12,pato14,lemu16}).
Specifically, we are interested in the asymptotic behavior of the Korn
constant $K(\GO_h,V_{h})$ for shells with zero Gaussian curvature as a function of their
thickness $h$ for subspaces $V_{h}$ of $W^{1,2}$ functions satisfying various
\bc s at the thin edges of the shell. In \cite{grtr07,grha16} we have shown that
$K(\Omega_h,V_{h})$ represents an absolute lower bound on safe loads for any slender
structure. For a classical circular cylindrical shell we have proved in
\cite{grha14} that $K(\Omega_h,V_{h})\sim h^{3/2}$ for a broad class of \bc s at the
thin edges of the shell.

The motivation for this work comes from the fact that the experimentally
measured buckling loads of axially compressed cylindrical shells behave in a
paradoxical way, dramatically disagreeing with predictions of classical shell
theory. The universal consensus is that such behavior is due to the extreme
sensitivity of shells to imperfections of shape and load. This study is a part
of rigorous analytical investigation of the influence of small changes in
shape on the structural behavior of cylindrically-shaped shells. It looks like
(and this will be addressed in future work) the determining factor of the
effect of shape imperfections is the Gaussian curvature of the actual imperfect
shell as the Ansatzen suggest in \cite{tosm01}. In this paper we show that if the
shell has a vanishing principal curvature (yielding zero Gaussian curvature), 
like circular cylindrical shells, then the scaling of the Korn constant 
$K(\GO_h,V_{h})$ will remain unaffected, provided the non-zero principal curvature
has a constant sign. Our analysis also shows that if both principal curvatures are zero on any open
subset of the shell's mid-surface, then $K(\GO_h,V_{h})\sim h^{2}$. We conjecture
that $K(\GO_h,V_{h})\sim h$ for shells of uniformly positive Gaussian curvature, while
$K(\GO_h,V_{h})\sim h^{4/3}$ if the Gaussian curvature is negative on any open subset
of the shell's middle surface, as suggested by test functions constructed in
\cite{tosm01}. These conjectures will be addressed elsewhere.

The goal of this paper is to show that the tools developed in \cite{grha14} for
circular cylindrical shells, and extended and developed further in \cite{haru14},
possess enough flexibility to be applicable to a
wide family of shells and in particular cylindrically-shaped shells (the ones, that
have no boundary in one of the principal directions).
The main idea is to first prove an inequality that is a hybrid between the first and second Korn inequalities
(we call it ``first-and-a-half Korn inequality'' for this reason) by
``assembling'' it from its two-dimensional versions corresponding to
cross-sections of the shell by curvilinear coordinate surfaces.
The first Korn inequality is then a consequence of the first-and-a-half
Korn inequality and an estimate on the normal component of $\Bu\in V_{h}$.  We
believe that this general methodology will work for broad classes of shells, 
even though the Gaussian curvature does affect the validity of some of the
technical steps in the proof, which must be adjusted to the particular case. 
In particular, the assumption of zero Gaussian
curvature is essential for all main results in Section~\ref{sec:3}.

\section{Preliminaries}
\setcounter{equation}{0}
\label{sec:2}
Consider a shell whose mid-surface is of class $C^2.$
Suppose $z$ and $\Gth$ are coordinates on the mid-surface of the shell, such
that $z=$constant and $\Gth=$constant are the lines of principal curvatures. Here
$\Gth$ will denote the circumferential coordinate and $z$--the longitudinal for 
cylindrically shaped shells. In the case of a straight circular cylinder 
$\Gth$ and $z$ are the standard cylindrical coordinates. Let
$\Br(\Gth,z)$ be the position vector of the shell's mid-surface. 
Introducing the normal coordinate $t$, we obtain the set of orthogonal
curvilinear coordinates $(t,\Gth,z)$, related to Cartesian coordinates via
\[
\Bx=\BR(t,\Gth,z)=\Br(z,\Gth)+t\Bn(z,\Gth),
\]
where $\Bn$ is the \emph{outward} unit normal, and $\BR(t,\Gth,z)$ is the
position vector of a point in space with coordinates $(t,\Gth,z)$. In this
paper we will study shells of uniform thickness $h$, given in $(t,\Gth,z)$
coordinates by
\begin{equation}
  \label{Ch}
  \CC_{h}=\{\BR(t,\Gth,z): z\in[L_{-},L_{+}],\ t\in I_{h},\ \Gth\in[0,p)\},\qquad
I_{h}=\left[-\frac{h}{2},\frac{h}{2}\right].
\end{equation}
We denote
\[
A_{z}^{2}=|\Br_{,z}|^{2},\qquad A_{\Gth}^{2}=|\Br_{,\Gth}|^{2}
\]
the two nonzero components\footnote{The principal directions are mutually
  orthogonal.} of the metric tensor of the middle surface. The two principal
curvatures will be denoted by $\Gk_{z}$ and $\Gk_{\Gth}$. Their signs are
chosen in such a way that $k_{z}$ and $k_{\Gth}$ are positive for a
barrel-shaped shells, like a sphere. The four functions $A_{\Gth}$, $A_{z}$,
$\Gk_{\Gth}$, and $\Gk_{z}$ satisfy the Codazzi-Gauss relations (see
e.g. \cite[Section~1.1]{tosm01})
\begin{equation}
  \label{Codazzi}
  \dif{\Gk_{z}}{\Gth}=(\Gk_{\Gth}-\Gk_{z})\frac{A_{ z,\Gth}}{A_{ z}},\qquad
\dif{\Gk_{\Gth}}{ z}=(\Gk_{z}-\Gk_{\Gth})\frac{A_{\Gth, z}}{A_{\Gth}},
\end{equation}
\begin{equation}
  \label{Gauss}
  \dif{}{ z}\left(\frac{A_{\Gth, z}}{A_{ z}}\right)+
\dif{}{\Gth}\left(\frac{A_{ z,\Gth}}{A_{\Gth}}\right)=-A_{ z}A_{\Gth}\Gk_{z}\Gk_{\Gth},
\end{equation}
and define the Levi-Civita connection on the
middle surface of the shell via the following derivation formulas
\[
\Grad_{\Be_{z}}\Be_{z}=-\nth{A_{ z}A_{\Gth}}\dif{A_{ z}}{\Gth}\Be_{\Gth}-\Gk_{z}\Bn,\qquad
\Grad_{\Be_{z}}\Be_{\Gth}=\nth{A_{ z}A_{\Gth}}\dif{A_{ z}}{\Gth}\Be_{z},\qquad
\Grad_{\Be_{z}}\Bn=\Gk_{z}\Be_{z},
\]
\[
\Grad_{\Be_{\Gth}}\Be_{\Gth}=-\nth{A_{ z}A_{\Gth}}\dif{A_{\Gth}}{ z}\Be_{z}-\Gk_{\Gth}\Bn,\qquad
\Grad_{\Be_{\Gth}}\Be_{z}=\nth{A_{ z}A_{\Gth}}\dif{A_{\Gth}}{ z}\Be_{\Gth},\qquad
\Grad_{\Be_{\Gth}}\Bn=\Gk_{\Gth}\Be_{\Gth}.
\]
Using these formulas we can compute the components of $\grad\Bu$ in the
orthonormal basis $\Be_{t}$, $\Be_{\Gth}$, $\Be_{z}$:
\[
\grad\Bu=
\begin{bmatrix}
  u_{t,t} & \dfrac{u_{t,\Gth}-A_{\Gth}\Gk_{\Gth}u_{\Gth}}{A_{\Gth}(1+t\Gk_{\Gth})} &
\dfrac{u_{t,z}-A_{z}\Gk_{z}u_{z}}{A_{z}(1+t\Gk_{z})}\\[3ex]
u_{\Gth,t}  &
\dfrac{A_{z}u_{\Gth,\Gth}+A_{z}A_{\Gth}\Gk_{\Gth}u_{t}+A_{\Gth,z}u_{z}}{A_{z}A_{\Gth}(1+t\Gk_{\Gth})} &
\dfrac{A_{\Gth}u_{\Gth,z}-A_{z,\Gth}u_{z}}{A_{z}A_{\Gth}(1+t\Gk_{z})}\\[3ex]
u_{ z,t}  & \dfrac{A_{z}u_{z,\Gth}-A_{\Gth,z}u_{\Gth}}{A_{z}A_{\Gth}(1+t\Gk_{\Gth})} &
\dfrac{A_{\Gth}u_{z,z}+A_{z}A_{\Gth}\Gk_{z}u_{t}+A_{z,\Gth}u_{\Gth}}{A_{z}A_{\Gth}(1+t\Gk_{z})}
\end{bmatrix}.
\]

We will now specialize to the particular case of zero Gaussian curvature
$\Gk_{z}=0$. In this case, equations
(\ref{Codazzi})--(\ref{Gauss}) can be solved explicitly in terms of four
arbitrary smooth functions $B(z)$, $a(\Gth)$, $b(\Gth)$, $c(\Gth)$:
\begin{equation}
  \label{As}
  A_{z}=B'(z),\qquad A_{\Gth}=a(\Gth)B(z)+b(\Gth),\qquad\Gk_{\Gth}=\frac{c(\Gth)}{A_{\Gth}}.
\end{equation}
We require that $A_{z}$, $A_{\Gth}$ and $c(\Gth)$ be strictly positive
functions of their variables on the mid-surface of the shell.
Hence, for shells of zero Gaussian curvature the formula for $\grad\Bu$
simplifies:
\begin{equation}
  \label{gradk0}
\grad\Bu=
\begin{bmatrix}
  u_{t,t} & \dfrac{u_{t,\Gth}-c(\Gth)u_{\Gth}}{A_{\Gth}+tc(\Gth)} &
\dfrac{u_{t,z}}{A_{z}}\\[3ex]
u_{\Gth,t}  &
\dfrac{u_{\Gth,\Gth}+c(\Gth)u_{t}+a(\Gth)u_{z}}{A_{\Gth}+tc(\Gth)} &
\dfrac{u_{\Gth,z}}{A_{z}}\\[3ex]
u_{ z,t}  & \dfrac{u_{z,\Gth}-a(\Gth)u_{\Gth}}{A_{\Gth}+tc(\Gth)} &
\dfrac{u_{z,z}}{A_{z}}
\end{bmatrix}.
\end{equation}
In the case of shells the thickness variable $t$ is uniformly small. We therefore
introduce the simplified gradient
\begin{equation}
  \label{gradsimp}
  \Grad\Bu=
\begin{bmatrix}
  u_{t,t} & \dfrac{u_{t,\Gth}-c(\Gth)u_{\Gth}}{A_{\Gth}} &
\dfrac{u_{t,z}}{A_{z}}\\[3ex]
u_{\Gth,t}  &
\dfrac{u_{\Gth,\Gth}+c(\Gth)u_{t}+a(\Gth)u_{z}}{A_{\Gth}} &
\dfrac{u_{\Gth,z}}{A_{z}}\\[3ex]
u_{ z,t}  & \dfrac{u_{z,\Gth}-a(\Gth)u_{\Gth}}{A_{\Gth}} &
\dfrac{u_{z,z}}{A_{z}}
\end{bmatrix}.
\end{equation}
We note in (\ref{gradsimp}) the components
$u_{t}$, $u_{\Gth}$ and $u_{z}$ are still functions of $(t,\Gth,z)$.

To be more specific we give two examples of zero Gaussian curvature shells:
cylinders and cones. A cylinder is described by a simple, smooth closed curve
of length $p$ in the $xy$-plane. Let $\BGr(\Gth)$, $\Gth\in[0,p)$ be the
position vector of this curve, parametrized by its arc-length. The position
vector of the middle surface of the shell is then given by
$\Br(\Gth,z)=\BGr(\Gth)+z\Be_{z}$, where $\Be_{z}$ is the unit vector
perpendicular to the $xy$=plane, i.e. the unit vector in the $z$-direction. It
is easy to verify that $\Gth=$constant and $z=$constant are lines of curvature
and $\Gk_{\Gth}=\Gk(\Gth)$ is the curvature of curve $\BGr(\Gth)$ in the
plane, whose sign is chosen to be positive for a circle.

A second example is a cone with vertex at the origin. A cone is described by a
a simple, smooth closed curve of length $p$ lying in the northern hemisphere of
a unit sphere centered at the origin. Let $\BGs(\Gth)$, $\Gth\in[0,p)$ be the
arc-length parametrization of this curve. In this case the middle surface of the shell is given by
$\Br(\Gth,z)=z\BGs(\Gth)$. Once again, it is easy to verify that
$\Gth=$constant and $z=$constant are lines of curvature and
\[
\Gk_{\Gth}=\frac{(\BGs(\Gth),\BGs'(\Gth),\BGs''(\Gth))}{z},
\]
where $(\Ba,\Bb,\Bc)=\Ba\cdot(\Bb\times\Bc)$ is the triple-product of 3
vectors in space. We summarize the data for cylinders and cones in Table~\ref{tab:cylcone}.
\begin{table}[h]
  \centering
  \begin{tabular}[h]{|c|c|c|c|c|}
    \hline
 & $B(z)$ & $a(\Gth)$ & $b(\Gth)$ & $c(\Gth)$\\
\hline
cylinders & $z$ & 0 & 1 & $\Gk(\Gth)$\\
\hline
cones & $z$ & 1 & 0 & $(\BGs(\Gth),\BGs'(\Gth),\BGs''(\Gth))$\\
\hline
  \end{tabular}
  \caption{Functions $B(z)$, $a(\Gth)$, $b(\Gth)$ and $c(\Gth)$ for cylinders and cones.}
  \label{tab:cylcone}
\end{table}

In this paper all norms $\|\cdot\|$ are $L^{2}$ norms. However because of
the curvilinear coordinates we will use several different flavors of the
$L^{2}$ inner product and the corresponding norm. For
$f,g\colon\CC_h\to\mathbb R$ we define the $L^{2}$ inner product
\[
(f,g)_{\CC_h}=\int_{\CC_h}f(\Bx)g(\Bx)d\Bx=\int_{I_{h}}\int_{l}^L\int_{0}^{p}A_zA_\Gth
fgd\Gth dzdt,
\]
which gives rise to the norm $\|f\|$.
\begin{equation}
  \label{normf}
  \|f\|^{2}=(f,f)=\int_{I_{h}}\int_{l}^L\int_{0}^{p}A_zA_\Gth f^{2}d\Gth dzdt.
\end{equation}
In cross-sections $\Gth$=constant we use
\begin{equation}
  \label{normth}
  \|f\|^{2}_{\Gth}=\int_{I_{h}}\int_{l}^{L}A_{z}f(t,\Gth,z)^{2}dzdt.
\end{equation}
We will also use the Euclidean version of the norm on cross-sections
\begin{equation}
  \label{eucnorm}
  \|f\|_{0}^{2}=\int_{\Ga_{1}}^{\Ga_{2}}\int_{\Gb_{1}}^{\Gb_{2}}f(\Ga,\Gb,\Gg)d\Ga d\Gb,
\end{equation}
where $\{\Ga,\Gb,\Gg\}=\{t,\Gth,z\}$ and $\Ga_{i}$, $\Gb_{i}$, $i=1,2$ are the
corresponding limits of integration. In each case it will be clear which
variable $t$, $\Gth$ or $z$ plays the role of the fixed variable $\Gg$. Of
course, due to uniform positivity of $A_{\Gth}$ and $A_{z}$ the norms
$\|f\|_{\Gg}$ and $\|f\|_{0}$ are obviously equivalent. In particular, all
inequalities involving one type of norm will also be valid for another.
Finally, all constants that are independent of $\Bu$ and $h$ will be denoted
by $C$. Once this is understood, such abuse of notation does not lead to any
ambiguity.

\section{Main results}
\setcounter{equation}{0}
\label{sec:3}
We formulate our Korn inequalities for vector fields $\Bu$ satisfying specific
\bc s at the two edges of the shell. We define
\begin{align}
\label{3.2}
V_{h}^1&=\{\BGf\in W^{1,2}(\CC_{h};\mathbb R^3):\BGf(r,\Gth,0)=\Bzr,\ \phi_{r}(r,\Gth,L)=\phi_{\Gth}(r,\Gth,L)=0\}\\ \nonumber
&\cap\{\BGf\in W^{1,2}(\CC_{h};\mathbb R^3): \BGf_\Gth,\BGf_z \ \text{are} \ p-\text{periodic in} \ \Gth\}.
\end{align}
and
\begin{align}
\label{3.3}
V_{h}^2&=\{\BGf\in W^{1,2}(\CC_{h};\mathbb R^3):\phi_{\Gth}(r,\Gth,0)=\phi_{z}(r,\Gth,0)=
\phi_{\Gth}(r,\Gth,L)=\phi_{z}(r,\Gth,L)=0\}\\ \nonumber
&\cap\{\BGf\in W^{1,2}(\CC_{h};\mathbb R^3): \BGf_\Gth,\BGf_z \ \text{are} \ p-\text{periodic in} \ \Gth\}.
\end{align}
We state our main results as a sequence of related theorems.
\begin{theorem}
\label{th:3.1}
Suppose $\Gk_{z}=0$ on $\CC_{h}$. Then there exist a constant $C$, independent
of $h$, such that for every $\Bu\in V_{h}^{1}$
\begin{equation}
\label{3.9}
\|\grad\Bu\|^2\leq C\left(\frac{\|u_t\|\|(\grad\Bu)_{\rm sym}\|}{h}+\|(\grad\Bu)_{\rm sym}\|^2\right),
\end{equation}
\begin{equation}
  \label{3.11}
\|\grad\Bu\|^2\leq \frac{C}{h^2}\|(\grad\Bu)_{\rm sym}\|^2,
\end{equation}
for all $h\in(0,1)$.
\end{theorem}
If additionally, the curvature $\kappa_\Gth$ vanishes on an open subset of the
middle surface of the shell then, according to Theorem~\ref{th:3.3} the bound
in (\ref{3.11}) is asymptotically sharp as $h\to 0$.
If the curvature $\kappa_\Gth$ does not change sign (i.e. stays uniformly
positive for cylindrically-shaped shells), then the first Korn inequality (\ref{3.11}) can be improved.
\begin{theorem}
\label{th:3.2}
Suppose that $\Gk_{z}=0$ and $\Gk_{\Gth}>0$ on
$\CC_{h}$. Then there exist a constant $C$, independent of $h$, such that for
every $\Bu\in V_{h}^1\cup V_{h}^2$ inequalities (\ref{3.9}) and
\begin{equation}
  \label{3.10}
\|\grad\Bu\|^2\leq \frac{C}{h\sqrt h}\|(\grad\Bu)_{\rm sym}\|^2,
\end{equation}
hold for all $h\in(0,1)$.
\end{theorem}
In fact, inequalities (\ref{3.9}) and (\ref{3.10}) are asymptotically sharp.
\begin{theorem}[Existence of optimal ansatz]~
\label{th:3.3}
\begin{enumerate}
\item[(i)] Suppose that assumptions of Theorem~\ref{th:3.1} are
  satisfied. Suppose, additionally that the curvature $\kappa_\Gth$ vanishes
  on an open subset of the middle surface of the shell. Then there exist
  functions $\Bu^{h}\in V_{h}^1\cap V_{h}^2$ and a constant $C$, independent
  of $h$, for which
  \begin{equation}
    \label{Kirch}
    \|(\grad\Bu^{h})_{\rm sym}\|^2\le Ch^{2}\|\grad\Bu^{h}\|^2.
  \end{equation}
\item[(ii)] Suppose that assumptions of Theorem~\ref{th:3.2} are
  satisfied. Then there exist functions $\Bu^{h}\in V_{h}^1\cap
V_{h}^2$ and a constant $C$, independent of $h$, for which
\begin{equation}
  \label{cyl}
  \|(\grad\Bu^{h})_{\rm sym}\|^2\le Ch\sqrt{h}\|\grad\Bu^{h}\|^2.
\end{equation}
\end{enumerate}
\end{theorem}

\begin{remark}
  Our results are formulated for shells cut along the coordinate
  surfaces. However, they are also valid for any shell $\CC_{h}$, bounded by
  the surfaces $z=Z_{\pm}(t,\Gth)$, where the spaces $V_{h}^{1}$, $V^{2}_{h}$
  are defined by (\ref{3.2}), (\ref{3.3}), respectively, except the indicated
  components of $\Bu$ vanish on the surfaces $z=Z_{\pm}(t,\Gth)$, instead of
  $z=L_{\pm}$.  This is because there exists shells
  $\CC_{h}^{-}\subset\CC_{h}\subset\CC_{h}^{+}$, such that the shells
  $\CC_{h}^{\pm}$ are bounded by surfaces $z=$constant. But then the ansatz
  from Theorem~\ref{th:3.3} supported in $\CC_{h}^{-}$ gives an upper bound on
  the Korn constant of $\CC_{h}$ that scales as $h^{3/2}$ (or as $h^{2}$). At
  the same time every function in $V_{h}^{1}$ or $V_{h}^{2}$ of $\CC_{h}$ can
  be extended (by extending the relevant components of $\Bu$ by zero) to a
  function in $V_{h}^{1}$ or $V_{h}^{2}$ of $\CC_{h}^{+}$, giving the lower
  bound on the Korn constant that scales as $h^{3/2}$ (or as $h^{2}$).
\end{remark}

\vspace{0.15cm}
\begin{remark}
Note, that the periodicity condition in the $\Gth-$direction in (\ref{3.2})-(\ref{3.3})
will be automatically satisfied for cylindrically-shaped shells. 
However, it is easy to check, that our main results hold (the proofs below go through) for different class or displacements $\Bu$
satisfying Robin boundary conditions on the thin faces of the shell. For instance, one can 
assume the condition 
\begin{equation}
\label{3.8}
\{\BGf\in W^{1,2}(\CC_{h};\mathbb R^3): \BGf_\Gth(t,0,z)=\BGf_\Gth(t,p,z)=0\}.
\end{equation}
on the thin edges of the shell in the $\Gth$ direction instead of the $p-$perioditicy of the $\Gth$ and the 
$z$ components in (\ref{3.2}) and (\ref{3.3}). Note, also, that the boundary condition (\ref{3.8}) is of 
Dirichlet type, but it can be rewritten as a Neumann condition $\BGF\cdot n=0$ on the thin edges of the 
shell in the $\Gth$ direction. 
\end{remark}

\section{Proofs of Theorems~\ref{th:3.1} and \ref{th:3.2}}
\setcounter{equation}{0}
\label{sec:4}
Our the strategy is to prove
first-and-a-half Korn inequality for the simplified version $\Grad\Bu$ of
$\grad\Bu$, given by (\ref{gradsimp})
\begin{equation}
\label{4.1}
\|\nabla\Bu\|^2\leq C\left(\frac{\|u_t\|\|e(\Bu)\|}{h}+\|e(\Bu)\|^2\right),
\end{equation}
where
\[
e(\Bu)=\hf\left(\Grad\Bu+(\Grad\Bu)^{T}\right).
\]
We then show that (\ref{4.1}) implies (\ref{3.9}). In order to prove
(\ref{4.1}) we apply the method, introduced in \cite{grha14}, of assembling
(\ref{4.1}) from the analogous two-dimensional inequalities corresponding to
the three coordinate surface cross-sections of the shell. Most of the proof is
done under the common assumptions of Theorems~\ref{th:3.1} and
\ref{th:3.2}. In other words, we assume that $\Bu\in V_{h}^{1}\cup V_{h}^{2}$
and we do not make any assumptions on the sign of $\Gk_{\Gth}$, until we say otherwise.

\subsection{The $t=\mathrm{const}$ cross-section}
The Korn-type inequality corresponding to $t=\mathrm{const}$ cross-section involves
$\Gth\Gth,$ $\Gth z$, $z\Gth$, and $zz$ components of the gradient. The
first-and-a-half Korn inequality in this case is stated in the following lemma.
\begin{lemma}
  \label{lem:tconst}
\begin{equation}
\label{4.2}
\|(\nabla\Bu)_{\Gth z}\|^2+\|(\nabla \Bu)_{z\Gth }\|^2\leq C\|e(\Bu)\|(\|e(\Bu)\|+\|u_t\|).
\end{equation}
\end{lemma}
\begin{proof}
Observing that
\[
\|(\nabla\Bu)_{\Gth z}\|^2+\|(\nabla \Bu)_{z\Gth }\|^2=
4\|e(\Bu)_{\Gth z}\|^{2}-2((\nabla\Bu)_{\Gth z},(\nabla \Bu)_{z\Gth })_{\CC_{h}},
\]
we conclude that it is sufficient to prove
\begin{equation}
\label{thz}
|((\nabla\Bu)_{\Gth z},(\nabla \Bu)_{z\Gth })_{\CC_{h}}|\leq C\|e(\Bu)\|(\|e(\Bu)\|+\|u_t\|).
\end{equation}
We have, that
\[
((\nabla\Bu)_{\Gth z},(\nabla \Bu)_{z\Gth })_{\CC_{h}}=\int_{I_{h}}\int_{l}^L
\int_{0}^{p}A_\Gth A_z\frac{u_{\Gth,z}}{A_z}\frac{u_{z,\Gth}-a(\Gth)u_\Gth}{A_\Gth}d\Gth dzdt
=\int_{I_{h}}(I_{1}(t)-I_{2}(t))dt,
\]
where
\[
I_1(t)=\int_{l}^L \int_{0}^{p}u_{\Gth,z}u_{z,\Gth}d\Gth dz,
\]
and
\[
I_2(t)=\int_{l}^{L}\int_{0}^p a(\Gth)u_\Gth u_{\Gth,z} d\Gth
dz=\hf\int_{0}^p a(\Gth)\int_{l}^{L}(u_{\Gth}^{2})_{,z}dz=0,
\]
since $u_{\Gth}=0$ at $z=l$ and $z=L$ in both $V^{1}_{h}$ and $V^{2}_{h}$.

Let us estimate $I_{1}(t)$. The idea is to observe that in both $V^{1}_{h}$ and $V^{2}_{h},$
\[
\int_{l}^L\int_{0}^{p}u_{\Gth,z}u_{z,\Gth}d\Gth dz=-
\int_{l}^L\int_{0}^{p}u_{\Gth}u_{z,\Gth z}d\Gth dz=
\int_{l}^L\int_{0}^{p}u_{\Gth,\Gth}u_{z,z}d\Gth dz,
\]
and then express $u_{\Gth,\Gth}$ and $u_{z,z}$ in terms of
$(\Grad\Bu)_{\Gth\Gth}=e(\Bu)_{\Gth\Gth}$ and $(\Grad\Bu)_{zz}=e(\Bu)_{zz}$,
respectively. Thus,
\begin{equation}
\label{4.3}
\int_{I_{h}}I_1(t)dt=\left(e(\Bu)_{zz},e(\Bu)_{\Gth\Gth}-\frac{a(\Gth)}{A_\Gth}u_z-\kappa_\Gth u_t\right).
\end{equation}
Applying the Schwartz inequality we obtain
\begin{align}
\label{4.4}
\left|\int_{I_{h}}I_1(t)dt\right|\leq C \|e(\Bu)\|(\|e(\Bu)\|+\|u_t\|+\|u_z\|).
\end{align}
By the Poincar\'e inequality
\begin{equation}
\label{4.5}
\|u_z\|\leq C\|u_{z,z}\|\leq C\|e(\Bu)\|.
\end{equation}
Using this inequality in (\ref{4.4}) we obtain the desired bound (\ref{4.2}).

\end{proof}

\subsection{The $\Gth=\mathrm{const}$ cross-section}
\begin{lemma}
\label{lem:Korntheta}
  \begin{equation}
\label{4.17}
\|(\nabla\Bu)_{tz}\|^2+\|(\nabla\Bu)_{zt}\|^2\leq C\left(\frac{\|u_t\|\cdot \|e(\Bu)\|}{h}+\|e(\Bu)\|^2\right).
\end{equation}
\end{lemma}
Let us show that Lemma~\ref{lem:Korntheta} is an immediate consequence of the
same two-dimensional inequality in Cartesian coordinates, proved in
\cite[Theorem~3.1]{grha14}. It states that if $h\in(0,1)$,
$\BGf=(u,v)\in H^{1}(I_{h}\times[L_{-},L_{+}];\bb{R}^{2})$ satisfies
$v(x,l)=0$, $x\in I_{h}$ in the sense of traces, then
 \begin{equation}
   \label{poltora}
   \|\nabla^{c}\BGf\|_{0}^{2}\le C\|e^{c}(\BGf)\|_{0}\left(\frac{\|u\|_{0}}{h}+\|e^{c}(\BGf)\|_{0}\right),
 \end{equation}
 where
\[
\nabla^{c}\BGf=\mat{u_{,x}}{u_{,y}}{v_{,x}}{v_{,y}},\qquad
e^{c}(\BGf)=\hf\left(\nabla^{c}\BGf+(\nabla^{c}\BGf)^{T}\right).
\]
Lemma~\ref{lem:Korntheta} says that the
same statement holds in our curvilinear coordinates, where $\Grad\Bu$ is
given by (\ref{gradsimp}).
\begin{proof}[Proof of Lemma~\ref{lem:Korntheta}]
We first prove
inequality (\ref{4.17}) for each fixed $\Gth\in[0,p]$:
\begin{equation}
  \label{K1.5theta}
  \|(\nabla\Bu)_{tz}\|^2_{\Gth}+\|(\nabla\Bu)_{zt}\|^2_{\Gth}\leq C\left(\frac{\|u_t\|_{\Gth}\cdot \|e_{\Gth}(\Bu)\|_{\Gth}}{h}+\|e_{\Gth}(\Bu)\|^2_{\Gth}\right),
\end{equation}
where
\[
\|f\|^{2}_{\Gth}=\int_{I_{h}}\int_{l}^{L}A_{z}f(t,\Gth,z)^{2}dzdt,\quad
e_{\Gth}(\Bu)=\mat{u_{t,t}}{\frac{1}{2}\left(\frac{u_{t,z}}{A_{z}}+u_{z,t}\right)}{\frac{1}{2}\left(\frac{u_{t,z}}{A_{z}}+u_{z,t}\right)}{\frac{u_{z,z}}{A_{z}}}.
\]
Then inequality (\ref{4.17}) is obtained from (\ref{K1.5theta}) by the
Cauchy-Schwartz inequality
\[
\int_{0}^{p}A_{\Gth}\|f\|_{\Gth}\|g\|_{\Gth}d\Gth\le\|f\|\|g\|,
\]
which becomes equality if $f=g$.
Let
\[
\BG=\mat{u_{t,t}}{u_{t,z}}{A_{z}u_{z,t}}{(A_{z}u_{z})_{,z}},\quad \BE=\hf(\BG+\BG^{T}).
\]
Then, by (\ref{poltora}), applied to $\BGf=(u_{t},A_{z}u_{z})$, we obtain
\begin{equation}
  \label{fromgrha14}
  \|\BG\|_{0}^{2}\le C\|\BE\|_{0}\left(\frac{\|u_{t}\|_{0}}{h}+\|\BE\|_{0}\right).
\end{equation}
By uniform positivity and boundedness of $A_{\Gth}$, $A_{z}$ norms
$\|\cdot\|_{0}$ and $\|\cdot\|_{\Gth}$ are equivalent. Hence,
\[
\|(\Grad\Bu)_{tz}\|_{\Gth}^{2}+\|(\Grad\Bu)_{zt}\|_{\Gth}^{2}\le
C(\|\BG_{tz}\|_{0}^{2}+(\|\BG_{zt}\|_{0}^{2})\le C\|\BG\|_{0}^{2}.
\]
Applying (\ref{fromgrha14}), we prove the lemma, if we show that
\begin{equation}
  \label{E0theta}
  \|\BE\|_{0}\le C\|e_{\Gth}(\Bu)\|_{\Gth}.
\end{equation}
We estimate
\[
\|\BE\|_{0}^{2}\le C(\|u_{t,t}\|_{\Gth}^{2}+\|u_{z,z}\|_{\Gth}^{2}+\|u_{z}\|_{\Gth}^{2})+
\hf\|u_{t,z}+A_{z}u_{z,t}\|_{0}^{2}.
\]
By the Poincar\'e inequality $\|u_{z}\|_{\Gth}^{2}\le \|u_{z,z}\|_{\Gth}^{2}$,
so that
\[
\|\BE\|_{0}^{2}\le
C\|e_{\Gth}(\Bu)\|_{\Gth}^{2}+\hf\|u_{t,z}+A_{z}u_{z,t}\|_{0}^{2}.
\]
It remains to observe that
\[
\|u_{t,z}+A_{z}u_{z,t}\|_{0}=\left\|\sqrt{A_{z}}\left(\frac{u_{t,z}}{A_{z}}+u_{z,t}\right)\right\|_{\Gth}\le C\|e_{\Gth}(\Bu)\|_{\Gth}.
\]
\end{proof}

\subsection{The  $z=\mathrm{const}$ cross-section}
\begin{lemma}
  \label{lem:zconst}
\begin{equation}
\label{4.25}
\|(\nabla\Bu)_{t\Gth}\|^2+\|(\nabla\Bu)_{\Gth t}\|^2 \leq
C\left(\frac{\|u_t\|\cdot\|e(\Bu)\|}{h}+\|e(\Bu)\|^2\right).
\end{equation}
\end{lemma}
\begin{proof}
  As before, we will show that that (\ref{4.25}) is a consequence of a
two-dimensional Korn-type inequality. However, before we can proceed with this
strategy, we observe that the term with $u_{z}$ in the $\Gth\Gth$-component of
$\Grad\Bu$ can be easily discarded due to the Poincar\'e inequality
(\ref{4.5}). Indeed, suppose we have proved (\ref{4.25}), where $e(\Bu)$ is
replaced with
\[
e'(\Bu)=\mat{e(\Bu)_{tt}}{e(\Bu)_{t\Gth}}{e(\Bu)_{t\Gth}}{e'(\Bu)_{\Gth\Gth}},
\]
where
\[
e'(\Bu)_{\Gth\Gth}=\dfrac{u_{\Gth,\Gth}+c(\Gth)u_{t}}{A_{\Gth}}.
\]
Then
\[
\|e'(\Bu)\|\le\|e(\Bu)\|+C\|u_{z}\|\le C\|e(\Bu)\|,
\]
so that
\[
\|(\nabla\Bu)_{t\Gth}\|^2+\|(\nabla\Bu)_{\Gth t}\|^2 \leq
C\left(\frac{\|u_t\|\cdot\|e'(\Bu)\|}{h}+\|e'(\Bu)\|^2\right)\le
C\left(\frac{\|u_t\|\cdot\|e(\Bu)\|}{h}+\|e(\Bu)\|^2\right).
\]
Next we prove the two-dimensional Korn-type inequality
\begin{lemma}
\label{lem:thetat}
  Let
\[
V=\left\{\BGf=(u,v)\in H^{1}_{\rm loc}\left(\left[-\frac{h}{2},\frac{h}{2}\right]\times\bb{R}\right):
u(x,y)\text{ is }p-\text{periodic}\right\}
\]
Let
\[
\BG=\mat{u_{,x}}{\Ga(y)u_{,y}-\Gk(y)v}{v_{,x}}{\Ga(y)v_{,y}+\Gk(y)u},\qquad
\BE=\hf(\BG+\BG^{T}).
\]
Then,
\begin{equation}
  \label{thetat}
  \|\BG\|_{0}^{2}\le C\left(\frac{\|u\|_{0}\|\BE\|_{0}}{h}+\|\BE\|_{0}^{2}+\|\BGf\|_{0}^{2}\right).
\end{equation}
\end{lemma}
This lemma is proved in Section~\ref{sec:Kt2D}.
If we apply this lemma to $\BGf(t,\Gth)=(u_{t}(t,\Gth,z),u_{\Gth}(t,\Gth,z))$
(for each fixed value of $z$) and then integrate over $z\in[L_{-},L_{+}]$, we obtain
the inequality (taking into account the equivalence between the curvilinear
norm (\ref{normf}) and the Euclidean norm)
\begin{equation}
  \label{thetat3d}
  \|(\nabla\Bu)_{t\Gth}\|^2+\|(\nabla\Bu)_{\Gth t}\|^2 \leq
C\left(\frac{\|u_t\|\cdot\|e(\Bu)\|}{h}+\|e(\Bu)\|^2+\|u_{t}\|^{2}+\|u_{\Gth}\|^{2}\right).
\end{equation}
By the Poincar\'e inequality and (\ref{4.2}) we obtain
\begin{equation}
  \label{utheta}
  \|u_{\Gth}\|^{2}\le C\|(\Grad\Bu)_{\Gth z}\|^{2}\le C\|e(\Bu)\|(\|e(\Bu)\|+\|u_t\|).
\end{equation}
At this point the different assumptions in Theorems~\ref{th:3.1} and
\ref{th:3.2} become important.

If $\Bu\in V^{1}_{h}$ then we can just use the Poincar\'e inequality and (\ref{4.17})
\begin{equation}
  \label{utbd1}
  \|u_{t}\|^{2}\le C\|(\Grad\Bu)_{tz}\|^{2}\le C\left(\frac{\|u_t\|\cdot\|e(\Bu)\|}{h}+\|e(\Bu)\|^2\right).
\end{equation}
Under the assumptions of Theorem~\ref{th:3.2} the Poincar\'e inequality cannot
be used. Instead we estimate expressing $u_{t}$ in terms of $e'(\Bu)_{\Gth\Gth}$:
\begin{equation}
  \label{uethth}
  u_{t}=\frac{A_{\Gth}e'(\Bu)_{\Gth\Gth}-u_{\Gth,\Gth}}{A_{\Gth}\Gk_{\Gth}}.
\end{equation}
Multiplying both sides of (\ref{uethth}) by $u_{t}$ and integrating by parts
(without writing integrals), we obtain
\[
u_{t}^{2}=\frac{u_{t}e'(\Bu)_{\Gth\Gth}}{\Gk_{\Gth}}-
\dif{}{\Gth}\left(\frac{u_{\Gth}u_{t}}{A_{\Gth}\Gk_{\Gth}}\right)+
u_{\Gth}u_{t}\dif{}{\Gth}\left(\frac{1}{A_{\Gth}\Gk_{\Gth}}\right)+
\frac{u_{\Gth}u_{t,\Gth}}{A_{\Gth}\Gk_{\Gth}}.
\]
Finally, replacing $u_{t,\Gth}$ in the formula above by
\[
u_{t,\Gth}=A_{\Gth}(\Grad\Bu)_{t\Gth}+c(\Gth)u_{\Gth}
\]
and using periodicity in $\Gth$ we obtain the estimate, taking into account
the equivalence of $\|\cdot\|_{0}$ and $\|\cdot\|$ norms
\[
\|u_{t}\|^{2}\le
C(\|u_{t}\|\|e(\Bu)\|+\|u_{t}\|\|u_{\Gth}\|+\|u_{\Gth}\|^{2}+\|(\Grad\Bu)_{t\Gth}\|\|u_{\Gth}\|).
\]
Hence, using the inequality $Cab\le a^{2}/2+C^{2}b^{2}/2$ (several times) and
replacing $\|u_{\Gth}\|^{2}$ by its estimate from (\ref{utheta}), we obtain
the bound.
\begin{equation}
  \label{utbd2}
  \|u_{t}\|^{2}\le C(\|e(\Bu)\|^{2}+\|(\Grad\Bu)_{t\Gth}\|\|u_{\Gth}\|).
\end{equation}
Using this inequality to estimate the term $\|u_{t}\|^{2}$ in (\ref{thetat3d})
we obtain the desired inequality (\ref{4.25}).

\end{proof}

\subsection{Conclusion of the proof}
Combining the estimates (\ref{4.2}), (\ref{4.17}) and (\ref{4.25}) we arrive
at (\ref{4.1}). However, $\Grad\Bu$ and $e(\Bu)$ are the simplified versions
of $\grad\Bu$ and $(\grad\Bu)_{\rm sym}$. Thus, we need to show that (\ref{3.9}) follows from
(\ref{4.1}). Under the assumptions of
Theorems~\ref{th:3.1} it is a consequence of
(\ref{utbd1}), while under the assumptions of
Theorems~\ref{th:3.2} it is a consequence of
(\ref{utbd2}).

The main observation in either case is that components of $\grad\Bu$ and
$\Grad\Bu$ are multiples of one another with coefficients that are independent
of $\Bu$. Thus, by direct calculation, we estimate
\begin{equation}
\label{4.33}
\|\grad\Bu-\nabla\Bu\|\leq Ch\|\grad\Bu\|,\quad\text{for all}\quad \Bu\in H^1(\CC_h),
\end{equation}
from which we get additionally,
\begin{equation}
\label{4.34}
\|(\grad\Bu)_{\rm sym}-e(\Bu)\|\leq \|\grad\Bu-\nabla\Bu\|\leq Ch\|\grad\Bu\|,\quad\text{for all}\quad \Bu\in H^1(\CC_h).
\end{equation}

\begin{proof}[Proof of Theorem~\ref{th:3.1}]
Observe, that $tt,$ $tz$, $zt$ and $zz$
components of $\grad\Bu$ and $\Grad\Bu$ coincide. The analysis for the cross
section $\Gth=\mathrm{const}$ involved only these components of the
gradient. Thus, estimate (\ref{4.17}) holds for $\grad\Bu$  in place of
$\Grad\Bu$. This implies (\ref{utbd1}) for $(\grad\Bu)_{\rm sym}$ in place of
$e(\Bu),$ i.e., we have
\begin{equation}
\label{4.35}
\|u_t\|\leq\frac{\|(\grad\Bu)_{\rm sym}\|}{h}.
\end{equation}
This allows us to show that (\ref{4.1}) implies (\ref{3.9}). Combining
 (\ref{4.1}), (\ref{4.33}) and (\ref{4.34}), we obtain
 \begin{equation}
   \label{almostKorn}
   \|\grad\Bu\|^2\leq C\left(\frac{\|u_t\|\|(\grad\Bu)_{\rm sym}\|}{h}+\|u_t\|\|\grad\Bu\|+\|(\grad\Bu)_{\rm sym}\|^2\right).
 \end{equation}
Estimating
\[
C\|u_t\|\|\grad\Bu\|\le\frac{1}{2}\|\grad\Bu\|^2+\frac{C^2}{2}\|u_t\|^2
\]
we get
\[
\|\grad\Bu\|^2\leq C\left(\frac{\|u_t\|\|(\grad\Bu)_{\rm sym}\|}{h}+\|u_t\|^{2}+\|(\grad\Bu)_{\rm sym}\|^2\right).
\]
Finally, by (\ref{4.35})
\[
\|u_t\|^{2}\le\frac{\|u_t\|\|(\grad\Bu)_{\rm sym}\|}{h},
\]
and (\ref{3.9}) follows.
Combining estimates (\ref{4.35}) and (\ref{3.9}) we obtain (\ref{3.11}).
\end{proof}

\begin{proof}[Proof of Theorem~\ref{th:3.2}]
In this case we proceed in the same way as in
\cite{grha14}, proving the following lemma.
\begin{lemma}
  \label{lem:K2}
Inequalities (\ref{4.1}), (\ref{utheta}), and (\ref{utbd2}) taken
together, imply
\begin{equation}
  \label{simpKorn}
  \|\Grad\Bu\|^{2}\le\frac{C}{h\sqrt{h}}\|e(\Bu)\|^{2}.
\end{equation}
\end{lemma}
We postpone the proof of this virtually algebraic lemma to finish the proof
of (\ref{3.9}).

Combining (\ref{simpKorn}) with (\ref{4.34}) we obtain
\[
\|\grad\Bu\|^{2}\le C\|\Grad\Bu\|^{2}\le
\frac{C}{h\sqrt{h}}(\|(\grad\Bu)_{\rm sym}\|^{2}+h^{2}\|\grad\Bu\|^{2}),
\]
proving the first Korn inequality (\ref{3.10}). Now, inequality (\ref{4.1}) and (\ref{3.10})
imply (\ref{3.9}). Indeed,
Using the estimate
\[
\|u_t\|\|\grad\Bu\|\le\frac{C\|u_t\|\|(\grad\Bu)_{\rm sym}\|}{h^{3/4}}\le\frac{C\|u_t\|\|(\grad\Bu)_{\rm sym}\|}{h}
\]
in (\ref{almostKorn}) we obtain (\ref{3.9}).

\end{proof}

\begin{proof}[Proof of Lemma~\ref{lem:K2}.] We begin with the inequality (\ref{utbd2}),
\[
\|u_{t}\|^{2}\le C(\|e(\Bu)\|^{2}+2\|\Grad\Bu\|\|u_{\Gth}\|)\le
C\left(\|e(\Bu)\|^{2}+\Ge^{2}\|\Grad\Bu\|^{2}+\nth{\Ge^{2}}\|u_{\Gth}\|^{2}\right)
\]
for any $\Ge>0$.
The small parameter $\Ge\in(0,1)$ will be chosen later to optimize the resulting
inequality. Estimating $\|u_{\Gth}\|^{2}$ by (\ref{utheta}) we obtain for
sufficiently small $\Ge,$
\[
\|u_{t}\|^{2}\le
C\left(\|e(\Bu)\|^{2}+\Ge^{2}\|\Grad\Bu\|^{2}+\nth{\Ge^{2}}(\|e(\Bu)\|^{2}+\|e(\Bu)\|\|u_{t}\|)\right).
\]
Estimating
\[
\frac{C}{\Ge^{2}}\|e(\Bu)\|\|u_{t}\|\le\hf\|u_{t}\|^{2}+\frac{C^{2}}{2\Ge^{4}}\|e(\Bu)\|^{2},
\]
we obtain
\[
\|u_{t}\|^{2}\le C\left(\frac{\|e(\Bu)\|^{2}}{\Ge^{4}}+\Ge^{2}\|\Grad\Bu\|^{2}\right).
\]
Thus,
\begin{equation}
\label{bound.on.u_r}
\|u_{t}\|\le C\left(\frac{\|e(\Bu)\|}{\Ge^{2}}+\Ge\|\Grad\Bu\|\right).
\end{equation}
Substituting this inequality into (\ref{4.1}), we obtain
\[
\|\Grad\Bu\|^{2}\le C\left(\frac{\|e(\Bu)\|^{2}}{h\Ge^{2}}+\frac{\Ge\|\Grad\Bu\|\|e(\Bu)\|}{h}\right).
\]
Estimating
\[
\frac{C\Ge\|\Grad\Bu\|\|e(\Bu)\|}{h}\le\hf\|\Grad\Bu\|^{2}+\frac{C^{2}\Ge^{2}\|e(\Bu)\|^{2}}{h^{2}},
\]
we obtain the inequality
\[
\|\Grad\Bu\|^{2}\le C\left(\frac{1}{h\Ge^{2}}+\frac{\Ge^{2}}{h^{2}}\right)\|e(\Bu)\|^{2}.
\]
We now choose $\Ge=h^{1/4}$ to minimize the upper bound and obtain (\ref{simpKorn}).
\end{proof}

\section{The two-dimensional Korn-type inequality}
\setcounter{equation}{0}
\label{sec:Kt2D}
In this section we prove the following two-dimensional Korn-type inequality.
\begin{theorem}
\label{th:thetat}
  Let
\[
V=\left\{\BGf=(u,v)\in H^{1}_{\rm loc}\left(\left[-\frac{h}{2},\frac{h}{2}\right]\times[0,p]\right): u(x,y)\text{ is }p-\text{periodic}\right\}
\]
and let
\begin{equation}
  \label{GEdef}
  \BG(\BGf)=\mat{u_{,x}}{a_{1}(y)u_{,y}+b_{1}(y)v}{v_{,x}}{a_{2}(y)v_{,y}+b_{2}(y)u},\qquad
\BE(\BGf)=\hf(\BG+\BG^{T}),
\end{equation}
where $a_{1}$, $a_{2}$, $b_{1}$, $b_{2}$ are Lipschitz continuous
functions, such that $a_{1}(y)$ and $a_{2}(y)$ do not vanish
on $[0,p]$.
Then,
\begin{equation}
  \label{thetat2}
  \|\BG\|^{2}\le C\left(\frac{\|u\|\|\BE\|}{h}+\|\BE\|^{2}+\|\BGf\|^{2}\right),
\end{equation}
where
\[
\|f\|^{2}=\int_{-h/2}^{h/2}\int_{0}^{p}f(x,y)^{2}dydx.
\]
\end{theorem}
The proof is based on the sharp inequality for harmonic functions
\cite[Lemma~4.3]{grha14}, \cite[Theorem~1.1]{haru14} which we formulate here for the sake of completeness.
\begin{lemma}
  \label{lem:4.1}
Let $R_h=(-\frac{h}{2},\frac{h}{2})\times(0,p).$
Suppose $w\in H^{1}(R_{h})$ is
harmonic in $R_{h}$ and satisfies one of the conditions:
\begin{itemize}
\item[(i)] $w(x,0)=w(x,p)$ or in the sense of traces,
\item[(ii)] $w(x,0)=0$ or in the sense of traces,
\item[(iii)] $w(x,p)=0$ or in the sense of traces.
\end{itemize}
 Then
\begin{equation}
  \label{4.22}
\|w_{y}\|^{2}\le\frac{2\sqrt{3}}{h}\|w\|\|w_{x}\|+\|w_{x}\|^{2}.
\end{equation}
\end{lemma}

\begin{remark}
For boundary conditions (\ref{3.8}), the $p-$periodicity of the $u$ component of $\BGf$ in Theorem~\ref{th:thetat} must be replaced by 
either of the condition $v(x,0)=0$ or $v(x,p)=0$ for all $x\in\left[-\frac{h}{2},\frac{h}{2}\right].$ Then one must apply Lemma~\ref{4.1}
to the extended version $\bar\BGf$ of the displacement $\BGf$ as follows:
$$u(x,-y)=u(x,-y), v(x,y)=-v(x,-y)\quad\text{for the condition}\quad v(x,0)=0,$$
$$u(x,2p-y)=u(x,-y), v(x,y)=-v(x,2p-y)\quad\text{for the condition}\quad v(x,p)=0.$$
Then the extended displacement $\bar\BGf$ satisfies the requirements of Theorem~\ref{th:thetat}.
\end{remark}
\begin{proof}[Proof of Theorem~\ref{th:thetat}]
  The first step is to replace $u(x,y)$ by its harmonic extension in $R_{h}$
  by defining $w\in H^{1}(R_{h})$ to be the unique solution
of the Dirichlet boundary value problem
\begin{equation}
\label{4.26}
\begin{cases}
\triangle w=0, &  (x,y)\in R_h\\
w=u, & (x,y)\in \partial R_h.
\end{cases}
\end{equation}
By the Poincar\'e inequality,
\begin{equation}
\label{4.27}
\|u-w\|\leq h \|\nabla(u-w)\|.
\end{equation}
Next, we express $\triangle(u-w)=\triangle u$ in terms of $\BE(x,y)$, defined
in (\ref{GEdef}), by eliminating all derivatives, except $u_{,y}$:
\begin{equation}
  \label{trianglewins}
  \triangle(u-w)=E_{11,x}+\frac{2E_{12,y}}{a_{1}}-\nth{a_{1}a_{2}}(E_{22,x}
+b_{2}E_{11}-b_{1}E_{22})+R(x,y),
\end{equation}
where
\[
R(x,y)=\frac{b_{1}b_{2}u-a_{2}a_{1}'u_{,y}-a_{2}b_{1}'v}{a_{1}a_{2}}.
\]
Now we multiply (\ref{trianglewins}) by $u-w$ and integrate by parts over
$R_{h}$ using the fact that $u-w$ vanishes on $\Md R_{h}$:
\[
\|\nabla(u-w)\|^{2}=E_{11}(u_{,x}-w_{,x})+\frac{2E_{12}(u_{,y}-w_{,y})}{a_{1}}-\frac{E_{22}(u_{,x}-w_{,x})}{a_{1}a_{2}}+(u-w)Q(x,y),
\]
where $Q(x,y)$ is a linear combination of $E_{11}$, $E_{12}$, $E_{22}$, $u$, $v$ and $u_{,y}$ with
uniformly bounded coefficients. Estimating $\|u-w\|$ by
(\ref{4.27}) we obtain, after division by $\|\nabla(u-w)\|$,
\begin{equation}
  \label{gradbd}
  \|\nabla(u-w)\|\leq Ch\left(\frac{\|\BE\|}{h}+\|u\|+\|v\|+\|u_{,y}\|\right).
\end{equation}
Our last task is to estimate $\|u_{,y}\|$. This is done by replacing $u$ with
$w$, estimating $\|w_{,y}\|$ using (\ref{4.22}), and returning back to
$u$, while controlling the incurred errors by (\ref{4.27}) and (\ref{gradbd}).
\begin{align*}
\|u_{,y}\|^2&\leq 2\|u_{,y}-w_{,y}\|^2+2\|w_{,y}\|^2
\leq C\left(\|\nabla(u-w)\|^2+\frac{\|w\|\|w_{,x}\|}{h}+\|w_{,x}\|^2\right)\\
&\leq C\left(\|\nabla(u-w)\|^2+\|u_{,x}\|^{2}+\nth{h}(\|u\|+h\|\nabla(u-w)\|)
(\|u_{,x}\|+\|\nabla(u-w)\|)\right)\\
&\leq C\left(\|\nabla(u-w)\|^2+\|u_{,x}\|^{2}+\frac{\|u\|\|u_{,x}\|}{h}
+\frac{\|u\|\|\nabla(u-w)\|}{h}\right)\\
&\le C\left(h^{2}\|u_{,y}\|^{2}+\|\BE\|^{2}+\frac{\|u\|\|\BE\|}{h}
+\|u\|^{2}+\|v\|^{2}+\|u\|\|u_{,y}\|\right),
\end{align*}
where we took into account that $u_{,x}=E_{11}$. Estimating
\[
C\|u\|\|u_{,y}\|\le\hf\|u_{,y}\|^{2}+\frac{C^{2}\|u\|^{2}}{2}
\]
and choosing $h$ so small that $Ch^{2}<1/4$ we obtain the inequality
\begin{equation}
  \label{uybd}
  \|u_{,y}\|^2\le C\left(\|\BE\|^{2}+\frac{\|u\|\|\BE\|}{h}+\|u\|^{2}+\|v\|^{2}\right),
\end{equation}
which holds for all sufficiently small $h>0$. Observing that
\[
\|\BG\|^{2}\le\|E_{11}\|^{2}+\|E_{22}\|^{2}+\|G_{12}\|^{2}+\|2E_{12}-G_{12}\|^{2},
\]
we get the bound
\[
\|\BG\|^{2}\le 7\|\BE\|^{2}+3\|G_{12}\|^{2},
\]
while
\[
\|G_{12}\|^{2}=\|a_{1}(y)u_{,y}+b_{1}(y)v\|^{2}\le C(\|u_{,y}\|^2+\|v\|^{2}).
\]
This shows that (\ref{uybd}) implies (\ref{thetat2}).
\end{proof}

\section{Proof of Theorem~\ref{th:3.3}}
\setcounter{equation}{0}

The ansatz in part (i) of the theorem is a classical Kirchhoff ansatz. The
assumptions of part (i) say that the shell contains a plate, which means that
we can introduce a local Cartesian coordinate system $(x_{1},x_{2},x_{3})$ in
which the (sub)plate be described as
\[
\CP_{h}=\{(x_{1},x_{2})\in\GO\subset\bb{R}^{2},\ x_{3}\in I_{h}\}.
\]
In these Cartesian coordinates we construct the ansatz in terms of the function
$\phi(x_{1},x_{2})$, compactly supported in $\GO$:
\begin{equation}
  \label{Kirchans}
  \begin{cases}
  u^{h}_{1}=-x_{3}\phi_{,x_{1}},\\
u^{h}_{2}=-x_{3}\phi_{,x_{2}},\\
u^{h}_{3}=\phi(x_{1},x_{2}).
\end{cases}
\end{equation}
Then
\[
\Grad\Bu=
\begin{bmatrix}
 -x_{3}\phi_{,x_{1}x_{1}} & -x_{3}\phi_{,x_{1}x_{2}} & -\phi_{,x_{1}}\\
 -x_{3}\phi_{,x_{1}x_{2}} & -x_{3}\phi_{,x_{2}x_{2}} & -\phi_{,x_{2}}\\
 \phi_{,x_{1}} & \phi_{,x_{2}} & 0
\end{bmatrix},\qquad
e(\Bu)=\begin{bmatrix}
 -x_{3}\phi_{,x_{1}x_{1}} & -x_{3}\phi_{,x_{1}x_{2}} & 0\\
 -x_{3}\phi_{,x_{1}x_{2}} & -x_{3}\phi_{,x_{2}x_{2}} & 0\\
 0 & 0 & 0
\end{bmatrix}.\qquad
\]
This shows that
\[
\|\Grad\Bu^{h}\|^{2}=\|\Grad\phi\|^{2}+\frac{h^{2}}{12}\|\Grad\Grad\phi\|^{2},\qquad
\|e(\Bu^{h})\|^{2}=\frac{h^{2}}{12}\|\Grad\Grad\phi\|^{2}.
\]
Choosing fixed non-zero $\phi\in C_{0}^{2}(\GO)$ we establish
(\ref{Kirch}). The ansatz (\ref{Kirchans}) was found by looking for the ansatz
in the form $\Bu^{h}=\Bv(x_{1},x_{2})+x_{3}\Bw(x_{1},x_{2})$. We then
compute $e(\Bu^{h})=\BE_{0}(x_{1},x_{2})+x_{3}\BE_{1}(x_{1},x_{2})$. The
ansatz (\ref{Kirchans}) is the general solution of the equations
$\BE_{0}(x_{1},x_{2})=0$. The same idea could be applied to $\grad\Bu$, given
by (\ref{gradk0}). However, the different structure of the gradient results
only in trivial solutions of $\BE_{0}(\Gth,z)=0$. Nevertheless we stick with
the same idea, looking for an ansatz in the form
$
\Bu^{h}=\Bv^{h}(\Gth,z)+t\Bw^{h}(\Gth,z).
$
Then,
\[
(\grad\Bu^{h})_{\rm sym}=\BE_{0}^{h}(\Gth,z)+t\BE_{1}^{h}(\Gth,z)+O(t^{2}).
\]
Now, instead of solving $\BE_{0}^{h}(\Gth,z)=0$ we demand that all components
of $\BE_{0}^{h}(\Gth,z)$ be zero, \emph{except the $zz$-component}.
In accordance with this strategy we have the following system of equations
\begin{equation}
  \label{ansyst}
  \begin{cases}
    w^{h}_{t}=0,\\
    w^{h}_{\Gth}=-\nth{A_{\Gth}}\left(\dif{v^{h}_{t}}{\Gth}+c(\Gth)v^{h}_{\Gth}\right),\\[2ex]
    w^{h}_{z}=-\nth{A_{z}}\dif{v^{h}_{t}}{z},\\[2ex]
    v^{h}_{t}=-\dfrac{v^{h}_{\Gth,\Gth}+a(\Gth)v^{h}_{z}}{c(\Gth)},\\[2ex]
    -A_{\Gth}v^{h}_{\Gth,z}=A_{z}(v^{h}_{z,\Gth}-a(\Gth)v^{h}_{\Gth}).
  \end{cases}
\end{equation}
The first four equations in (\ref{ansyst}) express $\Bw^{h}$ and $v^{h}_{t}$
in terms of $v^{h}_{\Gth}$ and $v^{h}_{z}$. The last equation relates
$v^{h}_{\Gth}$ and $v^{h}_{z}$, and needs to be solved. There are two mutually
exclusive cases
\begin{itemize}
\item Case 1:
\begin{equation}
  \label{sepvar}
  \frac{A_{z}}{A_{\Gth}}=\frac{H(\Gth)}{G(z)},
\end{equation}
for some Lipschitz functions $H(\Gth)$ and $G(z)$. It is easy to see from
formulas (\ref{As}) that
(\ref{sepvar}) is equivalent to $a(\Gth)$ and $b(\Gth)$ being linearly dependent,
i.e. there exists a constant scalar $\Gl_{0}$, such that either
$a(\Gth)=\Gl_{0}b(\Gth)$ or $b(\Gth)=\Gl_{0}a(\Gth)$.
\item Case 2: There exists an interval $I=(\Gth_{1},\Gth_{2})\subset(0,p)$,
  such that $a(\Gth)\not=0$ and $\rho'(\Gth)\not=0$ for all $\Gth\in I$, where
\[
\rho(\Gth)=\frac{b(\Gth)}{a(\Gth)}.
\]
\end{itemize}

\textbf{Case 1.} It is easy to see from Table~\ref{tab:cylcone} that all
cylinders and cones fall into this case. Under the assumption (\ref{sepvar})
the last equation in (\ref{ansyst}) has a general solution
\begin{equation}
  \label{gensol}
  v^{h}_{z}=A_{\Gth}G(z)H(\Gth)\phi^{h}_{,z},\qquad v^{h}_{\Gth}=-A_{\Gth}H(\Gth)^{2}\phi^{h}_{,\Gth},
\end{equation}
where $\phi^{h}(z,\Gth)$ can be an arbitrary function with compact support.

\textbf{Case 2.} In this case we will assume that functions $a(\Gth)$ and
$b(\Gth)$ are of class $C^{3}$. Solving the last equation in (\ref{ansyst})
with respect to $v^{h}_{z,\Gth}$
\[
v^{h}_{z,\Gth}=\nth{B'(z)}(a(\Gth)B'(z)v^{h}_{\Gth}-(a(\Gth)B(z)+b(\Gth))v^{h}_{\Gth,z})
\]
we see that we need both $a(\Gth)v^{h}_{\Gth}$ and $b(\Gth)v^{h}_{\Gth}$ to be
$\Gth$-derivatives of some $\Gth$-periodic smooth functions of
$(\Gth,z)$. Hence, we define
\[
v^{h}_{\Gth}=\frac{\psi^{h}_{,\Gth}}{a(\Gth)},
\]
where $\psi^{h}(\Gth,z)$ is supported on $I\times(L_{-},L_{+})$. But then, we also
need that $\rho(\Gth)\psi^{h}_{,\Gth}$ be a $\Gth$-derivative of some
$\Gth$-periodic smooth function of $(\Gth,z)$. We then define
\[
\psi^{h}=\frac{\phi^{h}_{,\Gth}}{\rho'(\Gth)},
\]
where $\phi^{h}(\Gth,z)$ is supported on $I\times(L_{-},L_{+})$. Then, we have
\[
\rho\psi^{h}_{,\Gth}=(\rho\psi^{h})_{,\Gth}-\rho'\psi^{h}=(\rho\psi^{h}-\phi^{h})_{,\Gth}.
\]
Hence,
\begin{equation}
  \label{vthz}
  \begin{cases}
  v^{h}_{\Gth}=\nth{a(\Gth)}\dif{}{\Gth}\left(\dfrac{\phi^{h}_{,\Gth}}{\rho'(\Gth)}\right),\\[2ex]
  v^{h}_{z}=\dfrac{B'(z)\phi^{h}_{,\Gth}+\rho'(\Gth)\phi^{h}_{,z}-
      (B(z)+\rho(\Gth))\phi^{h}_{,\Gth z}}{B'(z)\rho'(\Gth)}
  \end{cases}.
\end{equation}

Finally, in order to obtain optimal upper bound on the Korn constant we use
the same scaling analysis as in \cite{grha14} and define $\phi^{h}(\Gth,z)$ in
terms of the smooth, non-constant $p$-periodic in $\Gth$ function
$\Phi(\Gth,z)$. In Case 1 we just set
\[
\phi^{h}(\Gth,z)=\Phi(n(h)\Gth,z),
\]
where $n(h)$ is the integer part of $h^{-1/4}$. In Case 2 we define
\[
\phi^{h}(\Gth,z)=\eta(\Gth,z)\Phi(n(h)\Gth,z),
\]
where $\eta(\Gth,z)$ is a smooth $p$-periodic in $\Gth$ function, supported on $I\times(L_{-},L_{+})$.
In both cases the constructed ansatz
yields the upper bound $K(V_{h})\le Ch^{3/2}$ for any $V_{h}$ containing
$V_{h}^{1}\cap V^{2}_{h}$.

\section*{Acknowledgements.}
D.H. is grateful to Graeme Milton and the University of Utah for support.
This material is based upon work supported by the National Science Foundation
under Grants No. 1412058.

\bibliographystyle{abbrv}
\bibliography{refs}

\def\cprime{$'$} \ifx \cedla \undefined \let \cedla = \c \fi\ifx \cyr
  \undefined \let \cyr = \relax \fi\ifx \cprime \undefined \def \cprime
  {$\mathsurround=0pt '$}\fi\ifx \prime \undefined \def \prime {'}
  \fi\def\Ya{Ya}
\begin{thebibliography}{10}

\bibitem{ciar00}
P.~G. Ciarlet.
\newblock {\em Mathematical elasticity. {V}ol. {III}}, volume~29 of {\em
  Studies in Mathematics and its Applications}.
\newblock North-Holland Publishing Co., Amsterdam, 2000.
\newblock Theory of shells.

\bibitem{cot89}
D.~Cioranescu, O.~Oleinik, and G.~Tronel.
\newblock On {K}orn's inequalities for frame type structures and junctions.
\newblock {\em C. R. Acad. Sci. Paris S\'er. I Math.}, 309(9):591--596, 1989.

\bibitem{grha14}
Y.~Grabovsky and D.~Harutyunyan.
\newblock Exact scaling exponents in {K}orn and {K}orn-type inequalities for
  cylindrical shells.
\newblock {\em SIAM J. Math Anal.}, 46(5):3277--3295, 2014.

\bibitem{grha15}
Y.~Grabovsky and D.~Harutyunyan.
\newblock Rigorous derivation of the buckling load in axially compressed
  circular cylindrical shells.
\newblock {\em J. Elasticity}, 120(2):249--276, 2015.
\newblock to appear.

\bibitem{grha16}
Y.~Grabovsky and D.~Harutyunyan.
\newblock Scaling instability of the buckling load in axially compressed
  cylindrical shells.
\newblock {\em J. Nonlinear Sci.}, 26(1):83--119, 2016.

\bibitem{grtr07}
Y.~Grabovsky and L.~Truskinovsky.
\newblock The flip side of buckling.
\newblock {\em Cont. Mech. Thermodyn.}, 19(3-4):211--243, 2007.

\bibitem{haru14}
D.~Harutyunyan.
\newblock New asymptotically sharp {K}orn and {K}orn-like inequalities in thin
  domains.
\newblock {\em Journal of Elasticity}, 117(1):95--109, October 2014.

\bibitem{korn08}
A.~Korn.
\newblock Solution g{\'e}n{\'e}rale du probl{\`e}me d'{\'e}quilibre dans la
  th{\'e}orie de l'{\'e}lasticit{\'e}, dans le cas ou les efforts sont
  donn{\'e}s {\`a} la surface.
\newblock In {\em Annales de la facult{\'e} des sciences de Toulouse},
  volume~10, pages 165--269. Universit{\'e} Paul Sabatier, 1908.

\bibitem{korn09}
A.~Korn.
\newblock {\"U}ber einige {U}ngleichungen, welche in der {T}heorie der
  elastischen und elektrischen {S}chwingungen eine {R}olle spielen.
\newblock {\em Bull. Int. Cracovie Akademie Umiejet (Classe des Sci. Math.
  Nat.)}, pages 705--724, 1909.

\bibitem{lemu16}
M.~Lewicka and S.~Muller.
\newblock On the optimal constants in korn's and geometric rigidity estimates,
  in bounded and unbounded domains, under neumann boundary conditions.
\newblock {\em Indiana Math. Univ. Journal}, to appear.

\bibitem{love27}
A.~E.~H. Love.
\newblock {\em A treatise on the mathematical theory of elasticity}.
\newblock Dover, 4th edition, 1927.

\bibitem{naza04}
S.~A. Nazarov.
\newblock Weighted anisotropic {K}orn's inequality for a junction of a plate
  and a rod.
\newblock {\em Sbornik: Mathematics}, 195(4):553--583, 2004.

\bibitem{naza08}
S.~A. Nazarov.
\newblock Korn inequalities for elastic junctions of massive bodies, thin
  plates, and rods.
\newblock {\em Russian Mathematical Surveys}, 63(1):35, 2008.

\bibitem{oshyo92}
O.~A. Oleinik, A.~S. Shamaev, and G.~A. Yosifian.
\newblock {\em Mathematical problems in elasticity and homogenization},
  volume~26 of {\em Studies in Mathematics and its Applications}.
\newblock North-Holland Amsterdam, 1992.

\bibitem{pato12}
R.~Paroni and G.~Tomassetti.
\newblock Asymptotically exact {K}ornʼs constant for thin cylindrical domains.
\newblock {\em Comptes Rendus Mathematique}, 350(15):749--752, 2012.

\bibitem{pato14}
R.~Paroni and G.~Tomassetti.
\newblock On {K}orn's constant for thin cylindrical domains.
\newblock {\em Mathematics and Mechanics of Solids}, 19(3):318--333, 2014.

\bibitem{tosm01}
P.~E. Tovstik and A.~L. Smirnov.
\newblock {\em Asymptotic methods in the buckling theory of elastic shells},
  volume~4 of {\em Series on stability, vibration and control of systems}.
\newblock World Scientific, 2001.

\end{thebibliography}

\end{document}